\documentclass[11pt,a4paper]{amsart}
\usepackage{pdfsync}
\usepackage{mathptmx} % rm & math
\usepackage[scaled=0.90]{helvet} % ss
\usepackage{courier} % tt
\normalfont
\usepackage[T1]{fontenc}
\usepackage{version}
\usepackage{color}
\usepackage{amssymb}
\usepackage{bbold}
\usepackage{MnSymbol}
\usepackage{graphicx}
\usepackage{epstopdf}
\usepackage{enumitem}

\usepackage[curve,tips]{xypic}
\SelectTips{eu}{11}
\UseTips
\usepackage{hyperref}

\renewcommand{\epsilon}{\varepsilon}
\renewcommand{\setminus}{\smallsetminus}

\newtheorem{theorem}{Theorem}[section]
\newtheorem{proposition}[theorem]{Proposition}

\newtheorem{corollary}[theorem]{Corollary}
\newtheorem{lemma}[theorem]{Lemma}

\newtheorem*{cantorstheorem}{Cantor's Theorem}
\newtheorem*{theorema}{Theorem A}
\newtheorem*{theoremb}{Theorem B}

\newtheorem*{corollaryb1}{Corollary B1}
\newtheorem*{corollaryb2}{Corollary B2}
\newtheorem*{corollaryb3}{Corollary B3}

\newtheorem*{claim1}{Claim 1}
\newtheorem*{claim2}{Claim 2}
\newtheorem*{claim3}{Claim 3}
\newtheorem*{claim4}{Claim 4}
\newtheorem{example}[theorem]{Example}

\theoremstyle{definition}
\newtheorem*{definition*}{Definition}

\theoremstyle{remark}
\newtheorem{remark}[theorem]{Remark}
\newtheorem{remarks}[theorem]{Remarks}

\newcommand{\normal}{\lhd}

\newcommand{\Q}{\mathbb Q}
\newcommand{\Z}{\mathbb Z}
\newcommand{\N}{\mathbb N}

\newcommand{\C}{\mathbb C}
\newcommand{\F}{\mathbb F}

\newcommand{\mono}{\rightarrowtail}
\newcommand{\epi}{\twoheadrightarrow}

\renewcommand{\implies}{\Rightarrow}

\newcommand{\GL}{\operatorname{GL}}

\newcommand{\Maltsev}{{Mal${}'$\!cev}}

\newcommand{\dev}{\mathrm{dev}}
\newcommand{\krull}{\mathrm{Krull}}

\DeclareMathOperator{\toprk}{toprk}

\definecolor{amber}{rgb}{1.0, 0.75, 0.0}
\definecolor{violet}{rgb}{0.6, 0, 1.0}

\title{Soluble groups with no $\Z\wr\Z$ sections}
\begin{document}

\author[Jacoboni]{Lison Jacoboni}
\author[Kropholler]{Peter Kropholler}

\address{Lison Jacoboni, Laboratoire de Math\'ematiques d'Orsay, Univ. Paris-Sud, CNRS, Universit\'e Paris-Saclay, 91405 Orsay, France.}
\email{lison.jacoboni@math.u-psud.fr}
\address{Peter Kropholler, Mathematical Sciences, University of Southampton,  UK}
\email{p.h.kropholler@soton.ac.uk}

%\date{\today} % Activate to display a given date or no date

\subjclass[2010]{20J05, 20E22 }
\keywords{
wreath products, Krull dimension, soluble groups, torsion-free rank, 
\emph{
produits en couronne, dimension de Krull, groupes r\'esolubles, rang sans torsion}}

\thanks{L.~J. was supported by ANR-14-CE25-0004 GAMME}
\thanks{P.~K. was supported in part by  EPSRC grants no EP/K032208/1 and EP/ N007328/1.}
\thanks{The authors benefitted from a collaboration supported by the Isaac Newton Insitute}

\begin{abstract} In this article, we examine how the structure of soluble groups of infinite torsion-free rank with no section isomorphic to the wreath product of two infinite cyclic groups can be analysed. 
As a corollary, we obtain that if a finitely generated soluble group has a defined Krull dimension and has no sections isomorphic to the wreath product of two infinite cyclic groups then it is a group of finite torsion-free rank. There are further corollaries including applications to return probabilities for random walks. The paper concludes with constructions of examples that can be compared with recent constructions of Brieussel and Zheng.
\vspace{20pt}

\centerline{\it Groupes r\'esolubles sans section isomorphe \`a $\Z\wr\Z$}
\vspace{5pt}
\centerline{\bf R\'esum\'e}
\vspace{5pt}

{\it Dans cet article, nous donnons un th\'eor\`eme de structure pour les groupes r\'esolubles de rang sans torsion infini et sans section isomorphe au produit en couronne de deux groupes cycliques infinis. 
En cons\'equence, si un groupe r\'esoluble de type fini sans telle section admet une dimension de Krull alors il est de rang sans torsion fini. D'autres corollaires sont \'egalement d\'eduits, en particulier une application aux probabilit\'es de retour des marches al\'eatoires. L'article se termine avec la construction d'exemples qui peuvent \^etre compar\'es avec des travaux r\'ecents de Brieussel et Zheng.
}
\end{abstract}

\maketitle

\section{Introduction}

This paper examines the structure of finitely generated soluble groups of infinite torsion-free rank which have no sections isomorphic to $\Z\wr\Z$. The existence of such groups was established in \cite{K2}.  
Here we prove two theorems. Theorem A applies to any finitely generated soluble group of infinite torsion-free rank and shows in particular that there is a quotient $G$ of such a group that has an abelian Fitting subgroup $F$ with infinite torsion-free rank such that $G/F$ has finite torsion-free rank. These conclusions are drawn in Section 2 where we also establish some results on Krull dimensions of soluble and nilpotent groups. Theorem A also asserts that when the group has no $\Z\wr\Z$ sections then $G$ can be chosen to be residually finite. To prove this last part we require a more detailed structural result, Theorem B which is stated and proved in Section 3. Theorem B has a number of further  corollaries. For example, Corollary B1 asserts that finitely generated soluble groups with Krull dimension and no $\Z\wr\Z$ sections have finite torsion-free rank. We offer an interesting application to return probabilities for random walks on Cayley graphs for finitely generated soluble linear groups (see Corollary B3 below). In the concluding Section 4 of the paper we extend ideas of \cite{K2} to construct examples of groups with no $\Z\wr\Z$ sections and generalizations of lamplighter groups. We observe there is a very closed connection with recent construction of Brieussel and Zheng, see \cite{Brieussel,BZ}. In particular, both constructions produce groups that are abelian-by-(locally finite)-by-cyclic.

\section*{Acknowledgements}

The first author wishes to thank Yves de Cornulier and Romain Tessera for useful discussions.

\section{Notation, Background, and Theorem A}

\subsection{Classes of soluble groups} The terms of the derived series of a group $G$ are denoted $G^{(n)}$, inductively defined with $G^{(0)}=G$ and $G^{(n+1)}$ being the commutator subgroup $[G^{(n)},G^{(n)}]$. 
The soluble groups are those for which some term of the derived series is trivial and the derived length is the length of the derived series.
Recall that a group $G$ is \emph{soluble and minimax} provided it has a series
$\{1\}=G_0\triangleleft G_1\triangleleft\dots\triangleleft G_n=G$
in which the factors are cyclic or quasicyclic. By a \emph{quasicyclic} group, we mean a group $C_{p^\infty}$, where $p$ is a prime number, isomorphic to the group of $p$-power roots of unity in the field $\C$ of complex numbers. For a useful alternative point of view, the exponential map $z\mapsto e^{2\pi iz}$ identifies the additive group $\Z[\frac1p]/\Z$ with $C_{p^\infty}$. The terminology \emph{Pr\"ufer $p$-group} is often used to mean the quasicyclic group $C_{p^\infty}$.
For brevity, we write $\mathfrak M$ for the class of soluble minimax groups. 

For soluble groups the \emph{Hirsch length} or \emph{torsion-free rank} can be defined in terms of the derived series $(G^{(i)})$ by the formula $h(G)=\sum_{i\ge0}\dim_\Q G^{(i)}/G^{(i+1)}\otimes_\Z\Q$. Soluble groups of finite torsion-free rank possess a locally finite normal subgroup such that the quotient belongs to $\mathfrak M$. We write $\tau(G)$ for the largest normal locally finite subgroup of $G$.

Let $\mathfrak X$ denote the class of soluble groups of finite torsion-free rank, and let $\mathfrak X_q$ be the subclass of those having torsion-free rank $q$. Let $\mathfrak A_0$ denote the class of torsion-free abelian groups. For classes of groups $\mathfrak Y$ and $\mathfrak Z$, we write $\mathfrak{YZ}$ for the class of groups that have a normal $\mathfrak Y$-subgroup with corresponding quotient in $\mathfrak Z$. If $G$ is a finitely generated soluble group not in $\mathfrak X$ then it is a triviality to choose a quotient of $G$ that belongs to 
$\mathfrak A_0\mathfrak X$ but not $\mathfrak X$. For example, let $d$ be least such that $G^{(d-1)}/G^{(d)}$ has infinite torsion free rank and let $H/G^{(d)}$ be the torsion subgroup of  $G^{(d-1)}/G^{(d)}$. Then $G/H$ belongs to $\mathfrak A_0\mathfrak X\setminus\mathfrak X$. 
While it is not in general possible to choose a quotient that is \emph{just-non-$\mathfrak X$} we can nevertheless find quotients enjoying certain key properties. We write $\mathrm{Fitt}(G)$ for the join of the normal nilpotent subgroups of $G$. This is the Fitting subgroup. Fitting's lemma states that the join of two nilpotent normal subgroups is nilpotent and consequently the Fitting subgroup is locally nilpotent. It is the directed union of the nilpotent normal subgroups and it contains every subnormal nilpotent subgroup.

\begin{theorema}
Every finitely generated soluble group of infinite torsion-free rank  has a quotient $G$ with the following properties:
\begin{enumerate}
\item The Fitting subgroup $F$ of $G$ is torsion-free abelian of infinite rank, self-centralizing, and every non-trivial normal subgroup of $G$ meets $F$.
\item The factor group $G/F$ has finite torsion-free rank.
\item $\tau(G)$ is trivial.
\item If $K$ is a normal subgroup of $G$ then either $K$ is abelian-by-torsion or $G/K$ has finite torsion-free rank.
\end{enumerate}
Moreover, if the original group has no $\Z\wr\Z$ sections then 
every such $G$ is residually finite.
%\begin{enumerate}
%\setcounter{enumi}{5}
%\item 
%\end{enumerate}
%\item
%If $K$ is a normal subgroup of $G$ such that $KF/F$ is not locally finite then $G/K$ has finite torsion-free rank.
\end{theorema}

Part (i) should be compared with the standard fact that the Fitting subgroup of a just-infinite or just-non-polycyclic group is abelian. See \cite{RW} for further information. The last part of Theorem A concerning groups with no $\Z\wr\Z$ sections requires a further result, Theorem B, which is proved in Section 3. 

To prove the first part of Theorem A including items (i)---(iv) we need the following facts about nilpotent groups. Here and subsequently, $\zeta(K)$ denote the centre of the group $K$. We write $\gamma_i(K)$ for the $i$th term of the lower central series of $K$, that is $\gamma_1(K)=K$ and inductively $\gamma_{i+1}(K)=[\gamma_i(K),K]$. For nilpotent groups the lower series terminates in $1$ by definition and the \emph{class} of a nilpotent group is its length.

\begin{lemma}\label{lem2.1}
Let $K$ be a nilpotent group.
\begin{enumerate}
\item 
The set of elements of finite order in $K$ is a subgroup.
\item
If $K$ is torsion-free then so is $K/\zeta(K)$.
\item
If $K/\zeta(K)$ is torsion then $[K,K]$ is torsion.
\item
If $K$ is torsion-free and possesses an abelian normal subgroup $A$ such that $K/A$ is a torsion group then $K$ is abelian.
\end{enumerate}
\end{lemma}
\begin{proof}
Parts (i), (ii) and (iii) are standard results and we refer the reader to \cite{Rob96} for these and further background. For readers' convenience we include an argument to prove (iv).
Since $K$ is torsion-free, so is $K/\zeta(K)$ by (ii). By induction on class we may assume the result true of $K/\zeta(K)$ so we reduce at once to the case when $K/\zeta(K)$ is abelian. In this case, for any $a\in A$, the map $x\mapsto[x,a]$ is a homomorphism from $K$ to $\zeta(K)$. For any $x$ there is an $m\ge1$ with $x^m\in A$ and the homomorphism evaluates to $1$ this power of $x$. Since $K$ is torsion-free it follows that the homomorphism is trivial and hence $A$ lies in the centre of $K$, and $K/\zeta(K)$ (being a quotient of $K/A$) is torsion. Now (iii) implies that $[K,K]$ is torsion and since $K$ is torsion-free the result follows.
\end{proof}

The following further result about arbitrary groups is extremely important in analysing the structure of nilpotent groups.

\begin{lemma}\label{lem2.2}
Let $K$ be a group. Then for each $i$ there is a natural surjective homomorphism 
$$\underbrace{K/[K,K]\otimes\dots\otimes K/[K,K]}_i\to
\gamma_i(K)/\gamma_{i+1}(K).$$

In particular, if $K$ is nilpotent and $\mathfrak Z$ is an extension and quotient closed class of groups such that tensor products of abelian $\mathfrak Z$-groups are in $\mathfrak Z$ then $K$ belongs to $\mathfrak Z$ if and only if $K/[K,K]$ belongs to $\mathfrak Z$.
\end{lemma}
This result, including the application to classes of groups $\mathfrak Z$ with the stated closure properties is the content of Robinson's book: see \cite[5.2.5 and 5.2.6]{Rob96}. By taking $\mathfrak Z$ to be the class of nilpotent groups of finite torsion-free rank we deduce the

\begin{corollary}\label{cor2.3}
If $K$ is a nilpotent group of infinite torsion-free rank then $K/[K,K]$ has infinite torsion-free rank.
\end{corollary}

\begin{proof}[Proof of the first part of Theorem A]
We may replace the original group by a quotient $G$ that lies in $\mathfrak A_0\mathfrak X_q$, that has infinite torsion-free rank, and so that $q$ is as small as possible amongst quotients of $G$ with these two properties. Quotienting by $\tau(G)$ we may also assume that $\tau(G)=1$. Let $A$ be an abelian normal subgroup such that $G/A\in\mathfrak X_q$. Since $\tau(G)=1$ we see that $A$ is torsion-free.

If $N$ is a nilpotent normal subgroup of $G$ then $K:=NA$ is nilpotent of infinite torsion-free rank and hence $K/[K,K]$ has infinite torsion-free rank by Corollary \ref{cor2.3}. It follows that $G/K$ has torsion-free rank $q$ and so $K/A$ is torsion. Again, since $\tau(G)=1$, we have that $K$ is torsion-free. By Lemma \ref{lem2.1}(iv) $K$ is abelian and therefore $N$ is abelian. Hence every nilpotent normal subgroup of $G$ is abelian. Therefore the Fitting subgroup is abelian and the remaining assertions in (i) follow as in \cite[1.2.10]{LennoxRobinson}.
\end{proof}

\subsection{The set of rational numbers}The set of rational numbers has two roles in this paper. First it is the prime field of characteristic zero and we denote this by $\Q$. Secondly it is a countable dense linear order and when in this guise we denote it by $\mathbf Q$. In general a poset is a set with a reflexive antisymmetric and transitive relation $\le$. We shall write $x<y$ to mean $(x\le y$ and $x\ne y)$. We also freely use the notation $x>y$ and $x\ge y$ to mean $y<x$ and $y\le x$ respectively. The interval notation $[x,y]$ is used for the set
$\{z;\ x\le z\le y\}$.
The poset $\mathbf Q$ enjoys a special role on account of

\begin{cantorstheorem}[Theorem 9.3 of \cite{perms}]
Every countable dense linearly ordered set without endpoints is order-isomorphic to $\mathbf Q$.
\end{cantorstheorem}

We refer the reader to Chapter 9 of \cite{perms} for a careful introduction to Cantor's theorem and its ramifications.

\subsection{On deviation and Krull dimension}

We write $\dev(\mathcal S)$ for the deviation of a poset $\mathcal S$. 
The definition can be found in \cite{MR} and can be stated like this:
\begin{enumerate}
\item
$\dev(\mathcal S)=-\infty$ if $\mathcal S$ is \emph{trivial} (meaning that $a\le b\implies a=b$ for all $a,b\in S$).
\item
$\dev(\mathcal S)=0$ if $\mathcal S$ is non-trivial and artinian.
\item And in general by transfinite induction: $\dev(\mathcal S)$ is defined and equal to the ordinal $\alpha$ if $\mathcal S$ does not have defined deviation $\beta$ for any predecessor $\beta$ of $\alpha$ and, in every strictly descending chain $x_0>x_1>x_2>\dots$, all but finitely many of the intervals  $[x_{i+1},x_i]
%:=\{y\in\mathcal S;\ x_{i+1}\le y\le x_i\}
$ have deviation defined and preceding $\alpha$.
\end{enumerate}
For a group $G$ we write $\dev(G)$ for the deviation of the poset of subgroups of $G$ and we write $\krull(G)$ for the deviation of the poset of normal subgroups of $G$. This last is known as the \emph{Krull dimension} of $G$ \cite{Tus}. In ring theory, the Krull dimension of a module over a ring is defined to be the deviation of the poset of its submodules.

If a group $H$ acts on $M$, we write $\dev_H(M)$ for the deviation of $M$ as an $H$-group, which is the deviation of the subposet of subgroups of $M$ that are stable under the action of $H$.
Similarly, if a group $H$ acts on a group $M$ so that the action contains the inner automorphisms of $M$, we write $\krull_H(M)$ for the Krull dimension of $M$ as an $H$-group, defined as the deviation of the subposet of normal subgroups of $M$ that are stable under the action of $H$.

\begin{remark}
In this subsection, we shall visit two results, Lemmas \ref{lem-deviation nilpotent} and \ref{lem-deviation soluble} about deviation in nilpotent and soluble groups. In both cases there may be more general statements one could make using recent work of Cornulier, see \cite[Theorem 1.4]{Cor18}. 
\end{remark}

The next proposition studies how deviation and Krull dimension of $G$ behave with respects to extensions. It is stated in (\cite{J}, lemma 2.24) for the Krull dimension. The argument for the deviation is the same.

\begin{lemma}\label{lem-kdim g gps2}
Let 
\begin{equation*}
M \hookrightarrow G \overset{p}{\twoheadrightarrow} Q
\end{equation*}
be a sequence of $H$-groups.
Then,
\begin{align*}
\mathrm{Krull}_H (G) & = \max \{ \mathrm{Krull}_H (M), \mathrm{Krull}_H (Q)\}, \\
\dev_H(G) & = \max \{ \dev_H(M), \dev_H(Q)\}.
\end{align*}
\end{lemma}

\begin{lemma}
Let $\mathcal S$ be a poset. Then $\dev(\mathcal S)$ exists if and only if $\mathcal S$ has no subposet isomorphic to $\mathbf{Q}$. 
\end{lemma}
\begin{proof}
McConnell and Robson supply a statement and proof in \cite[Chapter 6, \S1.13]{MR}, but using the poset $\mathbf D:=\{m2^n;\ m,n\in\Z\}\cap[0,1]$ instead of $\mathbf Q$. Cantor's theorem allows us to reconcile theirs with ours since it implies that $\mathbf D$ minus endpoints is order-isomorphic to $\mathbf Q$. \end{proof}

The Krull dimension of a nilpotent group can be expressed in terms of the dimension of the factors of its lower central series, and is equal to its deviation.

\begin{proposition}\label{prop-Kr nilp}
Let $N$ be a nilpotent group.
Then,
\begin{equation*}
\krull(N) = \max_{1 \leqslant i \leqslant n} \left\{ \krull (\gamma_i(N) / \gamma_{i+1}(N)) \right\} =  \max_{1 \leqslant i \leqslant n} \left\{ \dev (\gamma_i(N) / \gamma_{i+1}(N)) \right\} = \dev(N) ,
\end{equation*}
where $n$ denotes the nilpotency class of $N$ and the groups $\gamma_i(N)$ form the lower central series of $G$.
\end{proposition}

\begin{proof}
Iterated applications of Lemma \ref{lem-kdim g gps2} yields
\[ \krull(N) = \max_{i = 1\dots n} \left\{ \krull_N(\gamma_i(N)/\gamma_{i+1}(N)) \right\} \]
and\[  \dev(N) = \max_{i = 1\dots n} \left\{ \dev(\gamma_i(N)/\gamma_{i+1}(N)) \right\}\]

%The proof is by induction on the nilpotency class. 
%The last term $N_n$ of the lower central series is abelian. Consider the short exact sequence 
%\[N_n \mono N \epi N/N_n. \]
%It follows from Lemma \ref{lem-kdim g gps2} that $\krull(N) = \max\{\krull_N(N_n), \krull_N(N/N_n)\}$.
%By induction,
%\[ \krull(N) = \max_{i = 1\dots n} \left\{ \krull_N(N_i/N_{i+1}) \right\}. \]
The action of $N$ on the factor $\gamma_i(N)/\gamma_{i+1}(N)$ is trivial. Indeed, let $n \in N$ and $n_i\gamma_{i+1}(N)$ an element of this factor, where $n_i \in \gamma_i(N)$. Then, $n\cdot n_i\gamma_{i+1}(N) = nn_in^{-1}\gamma_{i+1}(N) = n_i[n_i^{-1}, n]\gamma_{i+1}(N) = n_i\gamma_{i+1}(N)$ and
therefore we get the formula stated above.
\end{proof}

\begin{lemma}\label{lem-deviation nilpotent}
Let $G$ be a nilpotent group. Then the following are equivalent:
\begin{enumerate}
\item
$\dev(G)$ exists.
\item
$\dev(G)\le 1$.
\item
$\krull(G)$ exists.
\item
$\krull(G)\le 1$.
\item
$G$ is minimax.
\end{enumerate}
\end{lemma}
\begin{proof}
Proposition \ref{prop-Kr nilp} above and stability under extension of the minimax property imply that it is enough to prove the lemma for abelian groups. Hence, we may assume that $G$ is abelian. Note that $(i)$ and $(iii)$ are the same, as well as $(ii)$ and $(iv)$. We shall prove that $(v) \Rightarrow (ii) \Rightarrow (i) \Rightarrow (v)$. 

If $G$ is a minimax abelian group, then it is max-by-min. Finitely generated abelian groups have deviation $0$ or $1$ (\cite{J}, lemma 2.19), hence $G$ has deviation less or equal to $1$. 
It follows at once that $(ii) \Rightarrow (i)$, and $(i) \Rightarrow (v)$ is proved in (\cite{BCGS}, Lemma 4.6). 
\end{proof}
An alternative proof of $(i)\Rightarrow(v)$ can be devised by employing the variation on Corollary 2.3 that says $G$ is minimax if $G/[G,G]$ is minimax.

\medskip

\begin{remark}
Therefore, the Krull dimension of a nilpotent group $N$
\begin{itemize}
\item is $-\infty$ if $N = \{0\}$,
\item is $0$ if $N$ is non-trivial artinian,
\item is $1$ if $N$ is minimax non-artinian,
\item is not defined otherwise.
\end{itemize}
\end{remark}

The following similar lemma already appeared in the paper \cite{Tus} by Tushev. We provide a different proof. 

\begin{lemma}[\cite{Tus}]\label{lem-deviation soluble}
Let $G$ be a soluble group. Then the following are equivalent:
\begin{enumerate}
\item
$\dev(G)$ exists.
\item
$\dev(G)\le 1$.
\item
$G$ is minimax.
\end{enumerate}
\end{lemma}
\begin{proof}
If $G$ is abelian, this is the content of the previous lemma. If $G$ is soluble, the result follows by induction on its derived length.
\end{proof}

\medskip

The Krull dimension of metanilpotent groups can be expressed using particular module sections.

\begin{proposition}
Let $G$ be a metanilpotent group, that is
\begin{equation*}
N \hookrightarrow G \overset{p}{\twoheadrightarrow} P
\end{equation*}
where $N$ and $P$ are nilpotent. Then, 
\begin{equation*}
\krull(G) = \max\left\{\max\limits_{i = 1\dots n} \big\{\krull_{\Z P}(\gamma_i(N)/\gamma_{i+1}(N))\big\}, \max\limits_{j = 1\dots p} \big\{\krull(\gamma_i(P)/\gamma_{i+1}(P))\big\}\right\},
\end{equation*}
where $n$, resp. $p$, denote the nilpotency class of $N$, resp. $P$, and the groups $\gamma_i(N)$, resp. $\gamma_j(P)$, form the lower central series of $N$, resp. $P$.
\end{proposition}

\begin{proof}
Lemma \ref{lem-kdim g gps2} applied to the action of $G$ by conjugation yields \[\krull(G) = \max\{ \krull_G(N), \krull_G(P)\}.\]

First, note that the $G$-action on $P$ is actually a $P$-action and $\krull_G(P) = \krull(P)$. The desired formula for this last term is given in Proposition \ref{prop-Kr nilp}. Hence, we are left with studying $\krull_G(N)$.

Using the decomposition of $N$, we get
\begin{equation*}
\krull_G(N) = \max\limits_{i = 1 \dots n}\big\{ \krull_G(\gamma_i(N)/\gamma_{i+1}(N)) \big\}.
\end{equation*}
For brevity, write $\gamma_i$ for $\gamma_i(N)$. We claim that the $G$-action on $\gamma_i/\gamma_{i+1}$ induces an action of the quotient $P$, for $1 \leq i \leq n$. Indeed, let $x\gamma_{i+1}$ be an element of $\gamma_i/\gamma_{i+1}$ and $g, g'$ two elements of $G$ such that $g' = gn$, for some $n$ in $N$. We have
$g'\cdot x\gamma_{i+1} = g'x(g')^{-1} \gamma_{i +1} = g\cdot (nxn^{-1}\gamma_{i+1}) = g\cdot x\gamma_{i+1}$, where the last equality uses $nxn^{-1} = x[x^{-1}, n] \in x\gamma_{i+1}$. Hence, the action of an element of $G$ only depends on its image on the quotient $P$. Moreover, the groups $\gamma_i/\gamma_{i +1}$ are abelian groups, hence their dimension as $P$-groups is equal to their dimension as $\Z P$-modules. 

This proves the formula.
\end{proof}

\begin{remarks}\label{rem-kd apn npa}~

\begin{enumerate}
\item When $N$ is abelian (so that $G$ is abelian-by-nilpotent), we have
\begin{equation*}
\krull(G) = \max\left\{\krull_{\Z P}(N), \max\limits_{i = 1\dots n} \big\{\krull(\gamma_i(P)/\gamma_{i+1}(P))\big\}\right\}.
\end{equation*}
If $G$ is moreover finitely generated, $N$ is a finitely generated module over a Noetherian ring, hence is Noetherian. Indeed, this is a result, due to Hall \cite{Hall54} that the integral group ring of a polycyclic group is Noetherian. Hence, $G$ admits a Krull dimension. 
\item When $P$ is abelian (so that $G$ is nilpotent-by-abelian), we have
\begin{equation*}
\krull(G) = \max \left\{\krull(P), \max\limits_{i = 1 \dots n} \big\{ \krull_{\Z P}(\gamma_i(N)/\gamma_{i+1}(N)) \big\}\right\}.
\end{equation*}
If $G$ is moreover finitely generated, $\krull(P)$ is either $0$ or $1$ and the $\Z P$-modules $N_i/N_{i +1}$ are finitely generated, hence Noetherian. As a consequence, $G$ admits a Krull dimension. 
\end{enumerate}
\end{remarks}

\medskip

\begin{example}\label{Krull infinite tf rk and sections} 
Let $\Z$ act on the polynomial ring $\Z[X]$ by $f(X)*m:=f(X+m)$. Then
$\Z[X]\rtimes_*\Z$ is metabelian, locally nilpotent and has 
\begin{enumerate}
\item
Krull dimension 2,
\item
infinite torsion-free rank, and
\item
no section isomorphic to $\Z\wreath\Z$.
\end{enumerate}
\end{example}
\begin{proof}
Denote by $G$ the group $\Z[X]\rtimes_*\Z$.
\begin{enumerate}
\item The group $G$ is metabelian, we have the following short exact sequence $\Z[X] \mono G \epi \Z$.
Therefore, by (\cite{J}, proposition $2.24$) its Krull dimension is the Krull dimension of the $\Z\Z$-module $\Z[X]$. As the annihilator of $\Z[X]$ in the group ring $\Z\Z$ is trivial, this dimension equals $2$.
\item
$G$ has infinite torsion-free rank, as it contains $\Z[X]$ which is abelian of infinite torsion-free rank.
\item Existence of a section of $G$ isomorphic to $\Z\wr\Z$ would contradict local nilpotency.
\end{enumerate}
\end{proof}

\section{The Main Structure Theorem}

Let $\mathfrak V$ denote the class of finitely generated groups in
$\mathfrak A_0\mathfrak X\setminus\mathfrak X$ and let $\mathfrak U$ denote the class of those $\mathfrak V$-groups that have no $\Z\wreath\Z$ sections. Our goal in this section is to provide a description of the groups in $\mathfrak U$. Of course all such groups have a quotient satisfying all the conclusions of the Theorem A, but our structure theorem applies to arbitrary $\mathfrak U$-groups.

\medskip

For a $\Q Q$-module $V$, define the \emph{top rank} of $V$ to be the minimum over the dimensions of the irreducible quotients of $V$, that is
\[ \toprk_Q V = \min \{ \dim_\Q V/W \mid W \text{ is a maximal proper submodule of } V\}. \]

\begin{theoremb}\label{main}
Let $G$ be a $\mathfrak U$-group. 
Then $G$ has subgroups $A\subset K$ and $(A_j)_{j\in\N}$ such that the following hold:
\begin{enumerate}
\item All the subgroups $A_i$, $A$ and $K$ are normal.
\item $A$ is torsion-free abelian of infinite rank.
\item For each $j$, $A_j\subset A$, $A_j$ has finite rank, and $A$ is the direct product of the $A_j$.
\item $K/A$ is locally finite.
\item $G/K$ is a virtually torsion-free $\mathfrak M$-group.
\item For each $j$, $K/C_K(A_j)$ is finite.
\item For each subgroup $H$ of finite index in $K$ that is normal in 
$G$, $C_A(H)$ has finite rank.
\item For every $r$ in $\N$, the set $\{ j \in \N \mid \toprk_{G/A} (A_j\otimes \Q) \leq r \}$ is finite.
%For all but finitely many $j$, there is a subgroup $H_j$ of finite index in $G$ such that $H_j^{(2)}\cap A_j=\{1\}$ and $A_j$ satisfies the maximal condition $\mathrm{max}$-$G$ on subgroups that are normal in $G$.
\begin{comment}
\item The finite residual of $G$ is a divisible abelian \v{C}ernikov group.
\end{comment}
\end{enumerate}
\end{theoremb}

The following corollary should be compared with the main result of \cite{K84}: a finitely generated soluble group is either minimax or contains a section isomorphic to $\Z/p\Z\wr \Z$.

\begin{corollaryb1}\label{thm}
Let $G$ be a finitely generated soluble group with Krull dimension. Then $G$ has finite torsion-free rank if and only if $G$ has no sections isomorphic to $\Z\wreath\Z$.
\end{corollaryb1}
Example \ref{Krull infinite tf rk and sections} shows that finite generation is an essential hypothesis here.
\begin{proof}
The direct implication follows from the fact that $\Z \wr \Z$ has infinite torsion-free rank.

We prove the reverse implication. Let $G$ be a finitely generated soluble group with Krull dimension that has no section isomorphic to $\Z \wr \Z$.
We proceed by induction on the length of the derived series of $G$. Let $B$ be the last non-trivial term of this series. By induction, $G/B$ has finite torsion-free rank. We aim to prove that $B$ has finite torsion-free rank as well.

Quotienting out the torsion subgroup of $B$ has no impact on the torsion-free rank, thus we may assume that $B$ is a torsion-free abelian normal subgroup of $G$.

By contradiction, if $B$ were not of finite torsion-free rank, the group $G$ would belong to the class $\mathfrak{U}$.
Therefore, the description given in Theorem B would provide a torsion-free abelian normal subgroup $A$ of $G$, of infinite rank, such that $A$ is the direct product of infinitely many $\Z Q$-modules, where $Q$ stands for the quotient group $G/A$. Hence $A$ would have infinite uniform dimension, and consequently would not admit a Krull dimension, by \cite[6.2.6]{MR}. This contradicts the existence of the Krull dimension of $G$.
\end{proof}

\begin{corollaryb2}
Let $G$ be a finitely generated soluble group with no $\Z\wr\Z$ sections.
Then $G$ has finite torsion-free rank if and only if there is a finite bound on the torsion-free ranks of the metabelian-by-finite quotients of $G$. 
%$$\sup\{h(H/H^{(2)});\ H\le_{\mathit{fin}}G\}=\infty.$$ 
\end{corollaryb2}

\begin{proof}
We need to show that if the rank of $G$ is infinite then there are metabelian-by-finite quotients of arbitrarily large rank. Assume then that $G$ has infinite rank.
Theorem B applies and $G$ has a normal subgroup $A$ such that the quotient $Q = G/A$ belongs to $\mathfrak X$. The group $A$ is the direct product of infinitely many torsion-free abelian groups $A_j$ of finite rank. 

By $(viii)$, for every fixed integer $r$, one can find a $j$ with $\toprk_Q(A_j\otimes\Q) > r$. Hence there is a maximal proper $\Q Q$-submodule $W$ of $A_j\otimes Q$ such that  the rank of $A_j/(A_j\cap W)$ is greater than $r$.  Let $p_j:A\to A_j$ denote the projection and let $B$ denote the kernel of the composite $A\buildrel p_j\over\to A_j\to A_j/(A_j\cap W)$. Then $B$ is normal in $G$ and writing $\bar G$ for $G/B$ and $\bar A$ for $A/B$ we have a short exact sequence
%Set
%\[ \bar A = A /\left((A_j\cap W)\oplus \bigoplus_{i\neq j} A_i\right),\] it is a rationally irreducible module. The corresponding quotient $\bar G$ of $G$ satisfies the exact sequence
\[ \bar A \mono \bar G \epi Q. \]

Denote by $K$ the Fitting subgroup of $\bar G$. Let $L$ denote the isolator in $K$ of $K^{(1)}$, that is $L=\{x\in K;\ x^m\in K^{(1)}\textrm{ for some }m\in\N\}$.
There are two possibilities : either $L\cap \bar A = \{1\}$ or $\bar A \subset L$.

The first case produces a metabelian-by-finite quotient of $G$ with torsion-free rank greater than $r$. 

In the second case the quotient group $K/L$ is a section of $Q$.
Denote by $c$ the class of the Fitting subgroup of $Q$ and by $h$ the Hirsch length of its abelianization. The Hirsch length of $K/L$ is bounded above by $h$. A version of
Lemma \ref{lem2.2} for the isolator series applied to the group $K$ then implies that the rank of $A$ is bounded by $h^c$.
Therefore, it cannot happen when $r>h^c$.
\end{proof}

The theorem also has the following consequence for random walks on soluble linear groups. We refer for example to the survey of Tessera \cite{Tsurvey} for background and definitions. In \cite{J}, a group is said to have large return probability whenever its return probability is equivalent to $\exp(-n^\frac1{3})$.

\begin{corollaryb3}
Let $G$ be a finitely generated soluble linear group. Then either $G$ has large return probability or $G$ has a section isomorphic to $\Z\wr\Z$. 
\end{corollaryb3}

\begin{proof}
By results of \Maltsev \cite{Mal51} and Schur, the group $G$ has a finite index subgroup $H$ which is virtually torsion-free nilpotent-by-abelian. By Remark \ref{rem-kd apn npa}, $(ii)$, the group $H$, admits a Krull dimension. The dichotomy of Corollary $1$ above applies: either $H$ has finite torsion-free rank or $H$ has a section isomorphic to $\Z\wr \Z$. Therefore, either $G$ has finite rank or $G$ has a section isomorphic to $\Z\wr\Z$. The lower bound for the return probability of finitely generated soluble groups of finite rank is due to Pittet and Saloff-Coste \cite{PSC03}.
\end{proof}

\begin{proof}[Completion of the Proof of Theorem A]
Suppose now that $G$ is a finitely generated soluble group with the following properties. 
\begin{itemize}
\item The Fitting subgroup $F$ of $G$ is torsion-free abelian (of infinite torsion-free rank).
\item $G/F$ has finite torsion-free rank.
\item $\tau(G)$ is trivial
\item $G$ has no $\Z\wr\Z$ sections
\end{itemize}
Let $R$ denote the finite residual of $G$.
Let $A_j$ be a family of subgroups of $A$ as in the statement of Theorem B. For each $j$ let $A_j^*$ denote the direct sum of all $A_i$ with $i\ne j$. Then $G/A_j^*$ has finite torsion-free rank. Every finitely generated soluble group of finite torsion-free rank has a locally finite normal subgroup module which the group is minimax and residually finite,  hence the finite residual of $G/A_j^*$ is torsion.  In particular, $RA_j^*/A_j^*$ is torsion and it follows that $A_j\cap R$ is trivial for each $j$. From this it follows that $R\cap A$ is trivial. It therefore follows that $R$ is torsion and since $\tau(G)=1$, we have $R=1$ as required.
\end{proof}

\subsection*{Proof of the Structure Theorem}

\begin{lemma}\label{fin-dim}
Let $Q$ be a soluble group with a locally finite normal subgroup $K$ such that $Q/K$ is minimax. Let $M$ be a $\Q Q$-module on which $K$ acts trivially and which is locally finite dimensional. Then for any cohomology class $\xi\in H^n(Q,M)$ there exists a finite dimensional submodule $L$ of $M$ such that $\xi$ lies in the image of the map $H^n(Q,L)\to H^n(Q,M)$ induced by the inclusion of $L$ in $M$.
\end{lemma}
This is minor extension of Proposition $4$ in \cite{K84}. Example 3.2 below shows that the assumption that $K$ acts trivially on $M$ cannot be dropped.
\begin{proof}
Let $(M_j)_{j\in \N}$ be an ascending chain of finite-dimensional submodules that exhaust $M$. Triviality of the $K$-action implies that $M$ and the $M_j$'s are $\Q Q/K$-modules, and that 
\begin{align*}
H^q(K, M_j) = \begin{cases}
M_j \text{ if } q = 0, \\
0 \text{ otherwise,}
\end{cases} \text{ and } H^q(K, M) = \begin{cases}
M \text{ if } q = 0, \\
0 \text{ otherwise.}
\end{cases}
\end{align*}

Therefore the spectral sequence
\[ \lim_\rightarrow H^p(Q/K, H^q(K, M_j)) \Rightarrow \lim_\rightarrow H^{p+q}(Q, M_j)\] collapses when $q > 0$, and takes value $ \lim\limits_\rightarrow H^p(Q/K, M_j)$ when $q=0$.
We also have the following spectral sequence
\[ H^p(Q/K, H^q(K, M))  \Rightarrow  H^{p+q}(Q, M),\]
which collapses for $q > 0$ and takes value $H^p(Q/K, M)$ when $q = 0$.
In addition, there is a natural map from the first of these to the second. As a consequence, it is sufficient to prove that this natural map
\[ \lim_\rightarrow H^p(Q/K, M_j) \rightarrow  H^p(Q/K, M) \] is an isomorphism for all $p$.  As $Q/K$ is minimax and $M_j$ is finite dimensional, this follows from \cite[proposition $4$]{K84}.
\end{proof}

\begin{example}
Let $K = \F_2[t]$. For every natural number $j$, let $M_j$ be a finite dimensional $\Q K$-module on which 
\[ \bigoplus_{l \leqslant j} \F_2 t^j \]
acts trivially and such that $M_j^K =0$.  
As an example, one can take $\Q$ with the following $K$-action: for every $i$ in $\Z$ and every $x$ in $\Q$,
\[ t^i.x = \begin{cases}
- x \text{ provided } i = j + 1, \\
x \text{ otherwise.}
\end{cases} \]
We have the following exact sequence
\[ \bigoplus_{j\in\N} M_j \mono \prod_{j\in\N} M_j \epi X \]
where $X$ denotes the quotient of the product of the $M_j's$ by their direct sum. By construction, $K$ acts trivially on $X$. The long-exact sequence of cohomology
\[ \left(\bigoplus_{j\in\N} M_j\right)^K \mono \left(\prod_{j\in\N} M_j\right)^K \rightarrow X^K \rightarrow H^1\left(K, \bigoplus_{j\in\N} M_j\right) \rightarrow H^1\left(K, \prod_{j\in\N} M_j\right) \]
has many simplifications. First, note that the two left terms are trivial, and the third one is actually $X$. Moreover, by \cite{B},
\[ H^1(K, \prod M_j) = \prod H^1(K, M_j) =0\]
because each term of the product is zero.
Therefore, $H^1(K, \oplus M_j)$ is isomorphic to $X$, whereas the direct sum $\bigoplus_j H^1(K, M_j)$ is trivial. 
\end{example}

Let $G$ be a group, and $M$ a $\Z G$-module. We shall say that $M$ is a \emph{constrained module} if and only if for each $g \in G$, and $m \in M, m.\Z\langle g\rangle$ has finite abelian section rank. Similarly, if $k$ is a field, a $kG$-module will be called constrained if and only if for each $g \in G$ it is locally finite-dimensional as a $k\langle g\rangle$-module. For a given group ring, the class of constrained modules is both section and extension closed. These definitions were introduced in \cite{K84} by the second author, who proved that a finitely generated soluble group with no section isomorphic to $(\Z/p\Z) \wr \Z$ is minimax.

\begin{proposition}[\cite{K2}, Lemma 3.3]\label{1985}
Let $Q$ be a finitely generated soluble group of finite torsion-free rank and let $M$ be a constrained $\Q Q$-module. Then $M$ is locally finite dimensional.
\end{proposition}

\begin{lemma}\label{lem-direct sum}
Let $Q$ be a group, $T$ a normal subgroup of $Q$ and $V$ a $\Q Q$-module. Assume that $V$ is completely reducible as a $\Q T$-module. Denote by $\Lambda$ the set of isomorphisms classes of simple $\Q T$-submodules of $V$ and set, for every $\lambda \in \Lambda$ and for every orbit $\sigma \in \Lambda/Q$, 
\[ V_\lambda = \sum_{\substack{S \leqslant V \\ S\in\lambda}} S \text{\qquad and \qquad}W_\sigma = \sum_{\lambda \in \sigma} V_\lambda.\]
Then $W_\sigma$ is the $\Q Q$-submodule of $V$ generated by $V_\lambda$ and
\[ V = \bigoplus_{\sigma \in \Lambda/Q} W_\sigma. \]
\end{lemma}
\begin{proof}
By construction, $W_\sigma$ is a $\Q Q$-submodule of $V$ and $\sum W_\sigma = V$. Assume that $\sigma $ and $\gamma$ are such that $W_\sigma$ and $W_\gamma$ intersect non-trivially. Then this intersection contains a simple $\Q T$-submodule $S \simeq \lambda$ for some $\lambda \in \Lambda$ and $W_\sigma = W_\gamma$.
\end{proof}

We may now proceed to the proof of the Structure theorem.

\begin{proof}[Proof of the Structure Theorem]

As $G$ belongs to the class $\mathfrak{U}$, it has a normal torsion-free abelian subgroup $A$ with infinite torsion-free rank such that the quotient $Q= G/A$ is a finitely generated soluble group of finite torsion-free rank.
The group $Q$ has a locally finite normal subgroup $T = K/A$ such that the quotient $Q/T = G/K$ belongs to $\mathfrak{M}$. Let $\xi \in H^2(Q, A)$ be the cohomology class corresponding to the extension
\[ A \mono G \epi Q.\]

Denote by $V$ the tensor product $A \otimes \Q$. Since $G$ has no section isomorphic to $\Z \wr \Z$, it follows that $V$ must be a constrained $\Q Q$-module. Then, by proposition \ref{1985}, $V$ is locally finite dimensional.
\begin{claim1}
The module $V$ is a direct sum of simple $\Q T$-modules.
\end{claim1}
To prove this claim, consider
\[ \mathcal{X} = \{ X \subset V \mid X.\Q T \simeq \bigoplus_{x\in X} x.\Q T \text{ and for each } x \text{ in } X, x.\Q T \text{ is simple } \}. \]
Zorn's lemma provides a maximal element $X$ in $\mathcal{X}$. If $X.\Q T\neq V$, choose $v\in V\setminus X.\Q T$. By Maschke's theorem, $v.\Q T$ decomposes as a direct sum of simple $\Q T$-modules, and at least one of them is not contained in $X.\Q T$. Therefore, by changing the choice of $v$ if necessary, we may assume that $v.\Q T$ is simple. Set $X' = X \cup \{v\}$. The set $X'$ belongs to $\mathcal X$ and that is a contradiction. Therefore $X.\Q T = V$.

\medskip

Lemma \ref{lem-direct sum} applies and allows to write $V$ as the direct sum of the $\Q Q$-modules  $V = \oplus_\sigma W_\sigma$ where $\sigma$ runs along the orbits of the action of $Q$ on the set $\Lambda$ of isomorphisms classes of simple $\Q T$-submodules of $V$ and 
\[ W_\sigma = \sum_{\lambda \in \sigma} V_\lambda, \text{ where } V_\lambda = \sum_{\substack{S \leqslant V \\ S\in \lambda}} S.\]
\begin{claim2}
The $\Q Q$-modules $W_\sigma$ are finite dimensional over $\Q$.
\end{claim2}

Let $\lambda \in \sigma$ and $S\in \lambda$.  Consider the $\Q Q$-module $S.\Q G$: by local finiteness, it is finite dimensional over $\Q$. Set $T_0 = C_T(S.\Q G)$. The subgroup $T_0$ is normal in $Q$ and has finite index in $T$. It follows that the module $W_\sigma$ is acted on trivially by $T_0$, hence $W_\sigma \subset V^{T_0}$. Denote by $[V, T_0Â ]$ the span of the elements $v(t-1)$ for $v \in V$ and $t \in T_0$. It intersects $V^{T_0}$ trivially as $V/[V, T_0]$ is the biggest quotient of $V$ on which $T_0$ acts trivially. Therefore, the second claim follows from
\begin{claim3}
The quotient $V/[V, T_0]$ is finite dimensional.
\end{claim3}

\begin{comment}
Now, if $W$ is a simple $\Q T$-submodule of $V$, by local finiteness, $W.\Q G$ is finite dimensional over $\Q$. Set $T_0 = C_T(W.\Q G)$. The subgroup $T_0$ has finite index in $T$. It follows that the module 
\[ V_W = \sum_{M \simeq W.\Q G} M \] is acted on trivially by $T_0$ and therefore injective as  a $\Q T$-module by Lemma \ref{lem-injective}. Denote by $[V, T_0Â ]$ the span of the elements $v(t-1)$ for $v \in V$ and $t \in T_0$. It intersects $V_W$ trivially and furthermore $V/[V, T_0]$ is the biggest quotient of $V$ on which $T_0$ acts trivially. 
\end{comment}

To prove this third claim, first note that $Q/T_0$ is minimax. Let $\hat \xi$ be the image of $\xi$ in $H^2(Q, V/[V, T_0])$.
 Lemma \ref{fin-dim} provides a finite dimensional submodule $L/[V, T_0]$ of $V/[V, T_0]$ such that $\hat \xi$ lies in the image of the map $H^2(Q, L/[V, T_0]) \rightarrow H^2(Q, V/[V, T_0])$ induced by the inclusion of $L/[V, T_0]$ in $V/[V, T_0]$. Consequently, $\hat \xi$ goes to zero in $H^2(Q, V/L)$.  As $G$ is finitely generated and the extension 
\[ A/A\cap L \mono G/L\epi Q \]
splits, we have $V = L$ and therefore $V/[V, T_0]$ is finite dimensional.
This ends the proof of the claim.

\medskip

Consequently, $W_\sigma$ is finite dimensional and $V$ is a direct sum of finite dimensional $\Q Q$-modules.
Set $A_i = V_i \cap A$, this is a normal subgroup of $G$ with finite rank. The group $A$ contains the infinite direct sum of the $A_i$'s and the corresponding quotient $A/(\oplus A_i)$ is torsion. Hence, we may replace $A$ with $\oplus A_i$. This proves $(i)-(vi)$.

\medskip

The proof of $(vii)$ is similar to the third claim: if $H$ is a finite index subgroup of $K$, it acts trivially on $C_A(H)$ and similarly, $V/[V, H]$ is finite dimensional.

\medskip

To prove $(viii)$, fix $r \in \N$ and set $J_r:=\{j;\ d_j:=\toprk_Q(A_j\otimes\Q) \leq r\}$.
For every $j \in J_r$, there is a maximal proper $\Q Q$-submodule $W_j$ of $A_j\otimes
\Q$ such that $\bar A_j = A_j/(A_j \cap W_j)$ has rank less or equal to $r$.

Set

\[ W = \bigoplus_{j\in J_r} W_j,\qquad B = \bigoplus_{d_j \leq r} \bar A_j \text{\qquad and \qquad}  C = \bigoplus_{d_j > r} A_j. \]

Modulo $C\oplus W$, we obtain a quotient $\bar G$ of $G$ satisfying the extension
\[ B \mono \bar G \epi Q \] 

\begin{claim4}
There exists a subgroup $H$ of finite index in $\bar G$ such that $B \subset H$ and $H^{(1)}$ is nilpotent.
\end{claim4}

As $C_{\bar G}(\bar A_j)$ acts trivially on $\bar A_j$, $\bar G/C_{\bar G}(\bar A_j) \hookrightarrow GL_r(\Q)$ and there is a constant $C_r$ such that, for all $j\in J_r$, there exists a nilpotent-by-abelian normal subgroup $H_j$ of $\bar G$ such that $\bar A_j \subset H$, $[\bar G : H_j] \leq C_r$ and the subgroup $H_j^{(1)}$ acts nilpotently on $\bar A_j$. The intersection \[H = \bigcap_{j\in J_r} H_j\] still has finite index in $\bar G$, contains $B$ and $H^{(1)}$ acts nilpotently on $B$. Consequently, $H^{(1)}$ itself is nilpotent.

Assume that $J_r$ is infinite, then $H^{(1)}$ is a nilpotent group of infinite torsion-free rank. By Corollary \ref{cor2.3}, its abelianization $H^{(1)} /H^{(2)}$ also has infinite torsion-free rank. 

Therefore we obtain a metabelian quotient of $G$ with infinite torsion-free rank, contradicting the fact that $G$ does not admit a section isomorphic to $\Z\wr\Z$.

\begin{comment}
\bigskip

\begin{claim3} 
$V$ is a residually finite dimensional $\Q Q$-module.
\end{claim3}
To prove this thirs claim, fix $v$ in $V$ and consider a submodule $W$ of $V$, maximal with respect to not containing $v$. We show that $V/W$ is finite dimensional. 

If $Z$ is the submodule of $V$ generated by $v$ and $W$: $V/W$ is an essential extension of $Z/W$ and the latter is finite dimensional. 
Denote by $C$ the centralizer in $K$ of $Z/W$. By lemma \ref{red_box}, $C$ acts trivially on $V/W$. 
Besides, $K/C$ embeds in $GL_n(\Q)$ for some $n$ and is locally finite. Therefore it is finite (note that the same argument also proves (vi)) and $Q/C$ is minimax. 
Let $\hat \xi$ be the image of $\xi$ in $H^2(Q, V/W)$.
Lemma \ref{fin-dim} provides a finite dimensional submodule $L/W$ of $V/W$ such that $\hat \xi$ lies in the image of the map $H^2(Q, L/W) \rightarrow H^2(Q, V/W)$ induced by the inclusion of $L/W$ in $V/W$. Consequently, $\hat \xi$ goes to zero in $H^2(Q, V/L)$.  As $G$ is finitely generated and the extension 
\[ A/A\cap L \mono G/L\epi Q \]
splits, we have $V = L$ and therefore $V/W$ is finite dimensional.
This ends the proof of the claim.
\end{comment}
\end{proof}

\section{Explicit Examples}

Theorem B gives elaborate information about finitely generated soluble groups that have no $\Z\wr\Z$ sections that have infinite torsion-free rank. Here we give a general recipe for such groups. In fact examples can be found in early literature: their existence is made explicit in \cite{K2}. Examples are also present in more recent literature, for example in the work of Brieussel \cite{Brieussel} and Brieussel--Zhang \cite{BZ}, where their role is to provide examples of exotic analytic behaviour. 

%\subsection{A general strategy using lamplighter groups}
%
%In this section we describe a recipe for constructing finitely generated soluble groups of derived length $3$ that have no $\Z\wr\Z$ sections and yet have infinite torsion-free rank. Such groups always have a $C_p\wr\Z$ section for some prime $p$ by the main theorem of \cite{K84}.

\begin{definition*} We shall say that a group $Q$ is of \emph{lamplighter type} if it has all the following properties:
\begin{enumerate}
\item $Q$ is finitely generated.
\item $Q$ is  residually finite.
\item $Q$ has a normal locally finite subgroup $B$ and an element $t$ of infinite order such that $$Q=\bigcup_{i\in\Z}Bt^i.$$
\item $Q$ is not finite-by-cyclic.
\end{enumerate}
\end{definition*}

 Groups satisfying (i),  (iii) and (iv) do not necessarily satisfy (ii). For example, the (standard restircted) wreath product $S\wr\Z$ is never residually finite if $S$ is a (non-abelian) finite simple group although it has all the other properties of group of lamplighter type. The following fact is noteworthy and easily seen directly. We remark that the lemma also has a second proof based on our examples.

\begin{lemma}\label{lemmawith2proofs}
If a finitely generated group $Q$ satisfies property (iii) above then it is finite-by-cyclic if and only if it is finitely presented.
\end{lemma}
\begin{proof}[First proof]
Suppose  that $Q$ satisfies property (iii) with $B$ and $t$ as witnesses. If $Q$ is finitely presented then it is an HNN extension over a finite subgroup of $B$ by \cite[Theorem A]{BS} and the only possibility is that $B$ itself is finite and $Q$ is finite-by-cyclic. Conversely, all finite-by-cyclic groups are finitely presented.
\end{proof}

\begin{lemma}\label{setup}
Let $Q=\bigcup_{i\in\Z}Bt^i$ be a lamplighter-type group.
Then there is a sequence $(F_i,\ \Omega_i)$, ($i\ge0$) with the following properties.
\begin{enumerate}
\item The sequence
$$F_0<F_1<F_2<\dots$$
is an ascending chain of finite subgroups of $B$ such that $B=\bigcup_iF_i$.
\item each $\Omega_i$ is a finite $Q$-set that
\begin{itemize}
\item contains a $t$-fixed point,
\item is transitive as an $F_j$-set for $j>i$, and
\item is intransitive as an $F_j$-set for $j\le i$.
\end{itemize}
\end{enumerate}
\end{lemma}
\begin{proof}
Choose any strictly ascending chain of finite subgroups $F_i$ of finite subgroups of $B$ so that (i) holds. Using residual finiteness, we can find $N_i\normal Q$ such that $(B\cap N_i)F_i<B$. We now set $\Omega_i=B/B\cap N_i$ with $B$ acting by right multiplication and $t$ acting by conjugation. The first and third bullet points are automatically satisfied. 
The second bullet point may not hold but certainly, for any $i$, there is some $j>i$ such that $F_j$ acts transitively on $\Omega_i$ and we can simply replace the sequence $(F_i,\ \Omega_i)$ by a subsequence to ensure this is achieved with $j=i+1$.
\end{proof}

Now let $Q=\bigcup_{i\in\Z}Bt^i$ and $(F_i,\Omega_i)$ be as in Lemma \ref{setup}.
Define $M_i$ to be the kernel of the augmentation map $\Z\Omega_i\to\Z$ (given by $\omega\mapsto1$ for $\omega\in\Omega_i$).
Let $*_i$ denote a $t$-fixed point in $\Omega_i$. Let $S$ be a finite subset of $B$ such that $Q$ is generated by $\{t\}\cup S$.
%$\epsilon\colon\Z\Omega_i\to\Z$. We define a second map from $\eta:\Z\Omega_i\to\Z$ by
%$$*_i\mapsto0$$
%$$\omega\mapsto
%\begin{cases}
%0&\textrm{ if }\omega=*_i\\
%1&\textrm{ if }\omega\ne*_i.
%\end{cases}$$

Let $W_i$ and $X_i$ be distinct $F_i$-orbits in $\Omega_i$, one of which contains $*_i$. Set
$$\mathbf w_i:=\sum_{\omega\in W_i}\omega$$
$$\mathbf x_i:=\sum_{\omega\in X_i}\omega$$
$$\xi_i:=|X_i|\mathbf w_i-|W_i|\mathbf x_i.$$

We now have, for all $i$,
$$\xi_i\in M_i^{F_i}\eqno{\dagger}$$
and
$$\xi_i\notin\Z\Omega_i(t-1).\eqno{\ddagger}$$

Let $\mathbf\xi$ denote the element $(\xi_0,\xi_1,\dots)$ of $\overline M:=\prod_iM_i$
Recall that the semidirect product $Q\ltimes \overline M$ consists of ordered pairs $(q,\eta)$ where $q\in Q$, $\eta\in\overline M$ and the group multiplication is given by
$$(q,\eta)(q',\eta')=(qq'.\eta q'+\eta').$$
Inside this semidirect product
let $G$ be the subgroup generated by $\{\mathbf t\}\cup\mathbf S$ where
$$\mathbf t:=(t,\xi)$$
$$\mathbf S:=\{(s,0);s\in S\}.$$

\subsection*{Claim 1} Any word in $\{\mathbf t\}\cup\mathbf S$ that has exponent sum zero in $\mathbf t$ yields an element in the subgroup $B\ltimes M$ where $M=\bigoplus_iM_i$. 

\subsection*{Claim 2} $G$ has infinite torsion-free rank.

Claim 1 follows from $\dagger$ and claim 2 follows from $\ddagger$. We leave the details to the interested reader. The result is that we now have a finitely generated subgroup of $Q\ltimes\overline M$ which has derived length 3, has infinite torsion-free rank, and which has no $\Z\wr\Z$ sections.

\begin{proof}[Second proof of Lemma \ref{lemmawith2proofs} in case $Q$ is residually finite]
We have constructed a group $G$ that fits into a short exact sequence $$A\mono G\epi Q$$ where $Q$ is a group of lamplighter type by choosing a finitely generated subgroup of the semidirect product $Q\ltimes \overline M$. The image of the generating set $\{\mathbf t\}\cup\mathbf S$ in $Q$ generates $Q$. Moreover, the subgroup $A$ lies in $M$. In this situation, $Q$ cannot be finitely presented because $G$ is finitely generated and $A$ is plainly not finitely generated as normal subgroup.
\end{proof}

\begin{remark}
It is not clear just how much more general our concept of \emph{group of lamplighter type} is when compared to the standard lamplighter groups $\Z/p\Z\wr\Z$. However, by using the above construction while replacing the modules $M_i$ by the kernels of augmentation maps $\Z/p_i\Z[\Omega_i]\to\Z/p_i\Z$ for a sequence of non-zero integers $p_i$ we can build groups which are still of lamplighter type but clearly not of the classical form. It is perhaps worth remarking that groups $B$ that are simultaneously locally finite and residucally finite have strong structural restrictions: a just-infinite such group is constructed in \cite{BGS} where the authors also point out structural restrictions on such groups when they have finite exponent. Our construction yields new groups of lamplighter type where the base $B$ can have finite or infinite exponent.
\end{remark}

\begin{remark}
The group $G$ should be compared with the examples constructed in \cite{Brieussel} (and generalized in \cite{BZ}) to provide various behaviour of several characteristics associated to random walks on groups. The examples therein are diagonal products of finite lamplighter groups $D_l \wr\Z/m\Z$, where $D_l= \langle a, b \mid a^2=b^2=(ab)^l\rangle$ denotes the dihedral group of size $2l$. Just like the group $G$, they are $3$-step soluble groups. Brieussel constructs them as extensions of metabelian by abelian while our examples are constructed as abelian by metabelian. The end results are essentially the same in spirit.

More precisely, for $k \geq 0$,  consider the groups $D_l \wr\Z/m\Z$ with generating set $(+1, \mathbb{1}), (0, a\delta_0)$ and $(0, b\delta_{k})$, where $\mathbb 1$ is the identically neutral function and, for any $g\in D_l, \delta_g$ is the function taking value $1$ at $g$ and neutral elsewhere.

Let $(k_s), (l_s)$ and $(m_s)$ be three sequences of integers. The value $\infty$ is allowed for $(l_s)$ and $(m_s)$. In \cite{Brieussel}, Brieussel considered the infinite diagonal product $\Delta$ of the groups  $\Z/m_s\Z \wr D_{l_s}$  which is the subgroup of $\prod_s \Z/m_s\Z \wr D_{l_s}$ generated by the sequences $\big((+1, \mathbb 1)\big), \big((0, a_s\delta_0)\big)$ and $ \big((0, b_s\delta_{k_s})\big)$. The projection onto $\Z$ of an element of $\Delta$ does not depend on $s$, a feature shared with the group $G$.
\end{remark}

%\bibliography{JK}
%\bibliographystyle{abbrv}
%\bibliographystyle{alpha}
%\bibliographystyle{plain}
%\bibliographystyle{unsrt}

\end{document}